\documentclass[preprint,11pt,a4paper]{elsarticle}

\usepackage[T1]{fontenc}
\usepackage{lmodern}
\usepackage{amsmath,amssymb,amsthm,mathtools,booktabs}
\usepackage{microtype}
\usepackage{etoolbox}
\usepackage[a4paper,margin=22mm]{geometry}
\usepackage{hyperref}

\makeatletter
\def\rightharpoonfill@{\arrowfill@\relbar\relbar\rightharpoonup}
\newcommand\rbjt{\mathpalette{\overarrow@\rightharpoonfill@}}
\makeatother
\newcommand{\dist}{\operatorname{dist}}
\biboptions{numbers,sort&compress}

\hypersetup
{
colorlinks=true,
linkcolor=blue,
filecolor=blue,
urlcolor=blue,
citecolor=cyan,
}

\newtheorem{theorem}{Theorem}[section]
\newtheorem{lemma}[theorem]{Lemma}

\theoremstyle{definition}
\newtheorem{claim}{Claim}
\theoremstyle{plain}

\makeatletter
\renewenvironment{abstract}{%
  \global\setbox\absbox=\vbox\bgroup
  \hsize=\textwidth
  \noindent\hbox to \textwidth\bgroup\hfil
  \begin{minipage}{\dimexpr\textwidth-2\leftmargini\relax}
  \small
  \setlength{\parindent}{1.5em}%
  \noindent\hbox to \linewidth{\hfil\textbf{\@elsarticleabstitle}\hfil}%
  \par\medskip\indent\ignorespaces
}{%
  \par\end{minipage}\hfil\egroup
  \par\egroup
}
\patchcmd{\pprintMaketitle}{\hrule\vskip12pt}{\vskip12pt}{}{}
\patchcmd{\pprintMaketitle}{\hrule\vskip12pt}{\vskip12pt}{}{}
\makeatother

\renewenvironment{keyword}{%
  \def\sep{\unskip,\space}%
  \global\setbox\keybox=\vbox\bgroup
  \hsize=\textwidth
  \noindent\hbox to \textwidth\bgroup\hfil
  \begin{minipage}{\dimexpr\textwidth-2\leftmargini\relax}
    \small\normalfont
    \setlength{\parindent}{0pt}%
    \noindent\textit{Keywords:}\quad\ignorespaces
}{%
  \par\end{minipage}\hfil\egroup
  \par\egroup
}

\begin{document}

\begin{frontmatter}

\title{An asymptotically tight  bound on oriented diameter}

\author[nju]{Yaojun Chen}
\ead{yaojunc@nju.edu.cn}

\author[ecnu]{Jifu Lin}
\ead{jifulin01@163.com}

\author[fzu]{Xiaolin Wang\corref{cor1}}
\ead{xiaolinw92@163.com}

\author[nju]{Ruilin Zheng}
\ead{rlzhengmath@163.com}

\address[nju]{School of Mathematics, Nanjing University, Nanjing 210093, China}

\address[ecnu]{Department of Mathematics, East China Normal University, Shanghai 200241, China}

\cortext[cor1]{Corresponding author.}
\address[fzu]{School of Mathematics and Statistics, Fuzhou University, Fuzhou 350108, China}

\begin{abstract}
Let $f(d)$ be the smallest integer such that  every connected bridgeless graph of diameter $d$ admits an orientation with directed diameter at most $f(d)$. Chv\'atal and Thomassen (1978) proved  $\left\lceil{d^2}/{2}\right\rceil+d\le f(d) \le 2d^2+2d$.  
 In this paper, we prove  that  $f(d)\le \left\lceil d^2/2\right\rceil+6d$, which shows  
 $f(d)= d^2/2+\Theta(d)$ and determine the leading quadratic coefficient of $f(d)$.

The key method of our proof is to construct a  central subgraph $H$ such that it admits a strong orientation of diameter $O(d)$ and  the distance between $H$ and every vertex outside $H$ is at most $ d/2 $. 
\end{abstract}

\begin{keyword}Strong orientation\sep  Oriented diameter\sep   Diameter
\end{keyword}

\end{frontmatter}

\section{Introduction}
All graphs in this paper are finite and undirected unless an orientation has been specified. We allow parallel edges but no loops, and treat parallel edges as distinct edges. A graph is \emph{simple} if it has no parallel edges. For a connected graph $G$, the \emph{distance} $\dist_G(x,y)$ is the length of a shortest $x$--$y$ path, and its \emph{diameter} $\operatorname{diam}(G)$ is the maximum of these distances. An edge is a \emph{bridge} if its deletion disconnects the graph, and a graph is \emph{bridgeless} if it has no bridge. An \emph{orientation} $\rbjt{G}$ of $G$ assigns one direction to each edge independently. In particular, parallel edges may receive opposite directions. The orientation $\rbjt{G}$ is \emph{strong} if every ordered pair of vertices can be joined by a directed path. In 1939, Robbins~\cite{Robbins1939} proved that a connected graph has a strong orientation if and only if it is bridgeless. For a strong orientation $\rbjt{G}$ of $G$, let $\dist_{\rbjt{G}}(x,y)$ be the length of a shortest directed $x$--$y$ path, and let the \emph{directed diameter} $\operatorname{diam}(\rbjt{G})$ be the maximum directed distance. The \emph{oriented diameter} of a connected bridgeless graph $G$ is $\rbjt{diam}(G)=\min_{\rbjt{G}}\operatorname{diam}(\rbjt{G})$, where the minimum is over all strong orientations $\rbjt{G}$ of $G$.

By Robbins' theorem, a natural question is how small the diameter of a strong orientation can be. In particular, can it be bounded in terms of the diameter of the underlying graph? Clearly, $\rbjt{diam}(G)\ge diam(G)$. Chv\'atal and Thomassen~\cite{ChvatalThomassen1978} initiated the study of this question. For an integer $d\ge1$, let $f(d)$ be the smallest integer such that  every connected bridgeless graph of diameter $d$ admits an orientation with directed diameter at most $f(d)$. Equivalently, $$f(d)=\max\{\rbjt{diam}(G)~|~diam(G)=d~and~G~is~connected~and~bridgeless\}.$$ Chv\'atal and Thomassen~\cite{ChvatalThomassen1978} proved that
\(
\left\lceil\frac{d^2}{2}\right\rceil+d
\le f(d)\le 2d^2+2d,
\)
where the lower bound was given by a construction, and the upper bound was proved by the result of radius. The lower-bound construction for odd $d$ was subsequently corrected by Kwok, Liu and West~\cite{KwokLiuWest2010}.
Thus $f(d)=\Theta(d^2)$, but these estimates leave a factor of four between the leading coefficients. Therefore the main question  is  to determine the  coefficient of $d^2$. And this question remained the general benchmark for almost 50 years.

 In 2021, Babu, Benson, Rajendraprasad and Vaka~\cite{BabuEtAl2021} improved the upper bound for $d\ge8$ to
\(
f(d)\le1.373d^2+6.971d-1
\). More recently, Ai, Lei and Ren~\cite{AiLeiRen2026} improved the upper bound to
\(
f(d)\le d^2+\left\lceil\frac{11d}{2}\right\rceil-5.
\)
These results successively narrowed the gap to

\begin{equation}\label{1}
    \left\lceil\frac{d^2}{2}\right\rceil+d
\le f(d)\le d^2+\left\lceil\frac{11d}{2}\right\rceil-5.
\end{equation}

Alongside these general bounds, the exact values for small diameters have received considerable attention. Chv\'atal and Thomassen~\cite{ChvatalThomassen1978} proved $f(2)=6$. For diameter $3$, Kwok, Liu and West~\cite{KwokLiuWest2010} established $9\le f(3)\le11$, while Wang and Chen~\cite{WangChen2022} subsequently determined $f(3)=9$.
The next case, diameter $4$, has seen several recent advances~\cite{BabuEtAl2021,FengEtAl2026,LinYou2026,LinWangYou2026}. And the best known result of $f(4)$ is given very recently by Feng, Lei, Lian and Ren~\cite{FengEtAl2026}, who proved that $f(4)\le 16$.  Together with the  lower bound of \eqref{1}, we have  $12\le f(4)\le16$, leaving the exact value undetermined.
 These results show that determining $f(d)$ for an individual small value of $d$ requires substantially sharper estimates than the general quadratic bounds provide.

For simple graphs, oriented diameter has also been studied in terms of other graph parameters. Bounds involving the minimum degree were obtained by Bau and Dankelmann~\cite{BauDankelmann2015} and Surmacs~\cite{Surmacs2017}, with recent progress by Cochran and Wang~\cite{CochranWang2026}. Bounds involving the maximum degree, and the maximum degree together with girth, appear in~\cite{ChenChang2023,DankelmannGuoSurmacs2018}. For the domination number $\gamma(G)$, the minimum size of a vertex set containing a neighbor of every vertex outside it, Fomin, Matamala, Prisner and Rapaport~\cite{FominEtAl2004} proved $\rbjt{diam}(G)\le9\gamma(G)-5$. Kurz and L\"atsch~\cite{KurzLatsch2012} improved this to $4\gamma(G)$, and Wang and Chen~\cite{WangChen2025} established the sharp bound $\lceil(7\gamma(G)+1)/2\rceil$. Connected domination and distance domination number were considered by Dankelmann, Morgan and Rivett-Carnac~\cite{DankelmannEtAl2024}.

There is also an extensive literature on particular graph classes, including complete $k$-partite graphs for $k \geq 2$~\cite{G,LS,P}, AT-free graphs~\cite{FominEtAl2004}, chordal graphs~\cite{FominMatamalaRapaport2004},  maximal outerplanar graphs~\cite{WangEtAl2021},  near triangulations~\cite{GeLiuWang2025}   and triangulations~\cite{Liu2026,Wang}.  
In this paper, we prove the following theorem.
\begin{theorem}\label{thm:main}
Every finite connected bridgeless graph $G$ of diameter $d\ge1$ satisfies
\[
\rbjt{diam}(G)\le \left\lceil\frac{d^2}{2}\right\rceil+6d.
\]
\end{theorem}
Together with the lower bound of \eqref{1}, Theorem~\ref{thm:main} yields
\(
\left\lceil{d^2}/{2}\right\rceil+d
\le f(d)\le
\left\lceil{d^2}/{2}\right\rceil+6d.
\)
Consequently, $f(d)=d^2/2+\Theta(d)$.

The framework of our proof is inspired by the contraction method used  in~\cite{BabuEtAl2021,WangChen2025}.
 We first construct a subgraph $H$, which  is partly inspired by the proof for graphs of diameter~$3$ in~\cite{WangChen2022} by the first and third authors, where a single edge is selected and the vertex set is partitioned according to the distances from its two endpoints. 
In our construction of $H$, we replace the selected edge by two vertices $p$ and $q$ at distance $d$, partition $V(G)$ into two parts $X$ and $Y$ according to their relative distances from $p$ and $q$, and ensure that $H$ contains every edge between $X$ and $Y$. In this sense, $H$ acts as a central subgraph of $G$.
The two essential properties of $H$ are that every vertex outside $H$ is within distance $r=\lfloor d/2\rfloor$ of $H$, and $H$ admits a strong orientation of diameter at most $5d$. After contracting $H$ to a vertex $z$, the specified-root radius theorem of Chv\'atal and Thomassen~\cite{ChvatalThomassen1978} provides a strong orientation in which both directed distances between $z$ and every vertex are at most $r(r+1)$. Restoring $H$ and joining the resulting paths within its prescribed orientation yields a strong orientation of $G$ with diameter as required.

We finish this section with the notation used below.  A path has no repeated vertex, whereas a walk may repeat vertices; its length is its number of edges or arcs, and paths of length zero are allowed.  For a path $Q$, the notation $|Q|$ means its length. The radius of a connected graph $G$ is $\min_z\max_x\dist_G(z,x)$.    An orientation is acyclic if it contains no directed cycle. A \emph{cut vertex} is a vertex whose deletion disconnects a connected graph. A  graph with at least two vertices is \emph{2-connected} if every two distinct vertices can be connected by two internally disjoint paths. A \emph{block} of a  graph is a maximal 2-connected subgraph, a bridge together with its ends, or an isolated vertex.  Then every cycle lies in a single block.  

In the rest of this paper,
Section~\ref{sec2} constructs $H$ and its strong orientation, with a linear bound of the directed distances. Section~\ref{sec:distance} proves that every vertex outside $H$ is within distance $\lfloor d/2\rfloor$ of $H$. Section~\ref{sec:contraction} proves Theorem~\ref{thm:main} by contracting $H$ and applying the radius theorem.

\section{A central subgraph and its orientation}\label{sec2}

The aim of this section is to prove Lemma~\ref{lem:core}, which supplies the central  subgraph and the directed-distance bound.
Let $G$ be a finite connected bridgeless graph of diameter $d\ge1$. Choose $p,q\in V(G)$ with $\dist_G(p,q)=d$, and write $a(x)=\dist_G(p,x)$ and $b(x)=\dist_G(q,x)$ for every $x\in V(G)$. Let $\varepsilon=1$ when $d$ is even and $\varepsilon=0$ when $d$ is odd. Define
\[
X=\{x\in V(G):a(x)-b(x)\le\varepsilon\},\qquad Y=V(G)\setminus X.
\]
Since $a(p)-b(p)=-d$ and $a(q)-b(q)=d>\varepsilon$, we have $p\in X$ and $q\in Y$. An edge with one endpoint in each of $X$ and $Y$ is called a \emph{crossing edge}.

\begin{lemma}\label{lem:core}\label{lem:H}\label{lem:block}
Define $\varphi(x)=a(x)$ for $x\in X$ and $\varphi(x)=2d-\varepsilon-b(x)$ for $x\in Y$. There is a connected bridgeless subgraph $H$ of $G$ containing every crossing edge and admitting a strong orientation $\rbjt{H}$ of $H$ such that
\begin{equation}\label{eq:core}
\dist_{\rbjt{H}}(u,v)\le \varphi(v)-\varphi(u)+3d-\varepsilon
\qquad(u,v\in V(H)).
\end{equation}
\end{lemma}

The proof of this lemma first constructs an auxiliary subgraph $K$ and its acyclic orientation, then finds a connected bridgeless subgraph $H$ of $K$, and finally orients $H$ and proves~\eqref{eq:core}. Before proving Lemma~\ref{lem:core}, we prove the following lemma which will be used in the final step.

\begin{lemma}\label{lem:reversal}
Let $J$ be a connected bridgeless graph, and let $\rbjt{J}$ be an acyclic orientation of $J$. Suppose that $s\ne t$ and every vertex of $J$ lies on a directed $s$--$t$ path in $\rbjt{J}$. Let $Q$ be an arbitrary directed $s$--$t$ path in $\rbjt{J}$, and let $\rbjt{J_Q}$ be the orientation of $J$ obtained from $\rbjt{J}$ by reversing only the arcs of $Q$. Then $\rbjt{J_Q}$ is strong.

Moreover, if an integer-valued function $\varphi$ on $V(J)$ satisfies $\varphi(v)-\varphi(u)\ge1$ for every arc $(u,v)$ of $\rbjt{J}$, then for any two distinct vertices $x,y$,
\(
\dist_{\rbjt{J_Q}}(x,y)\le \varphi(y)-\varphi(x)+\varphi(t)-\varphi(s)+|Q|.
\)
\end{lemma}
\begin{proof}
First, we claim that every arc $(u,v)$ of $\rbjt{J}$ lies on a directed $s$--$t$ path. Indeed, the hypothesis supplies a directed $s$--$u$ path and a directed $v$--$t$ path. Their concatenation with $(u,v)$ is a directed walk. If this walk repeated a vertex, it would contain a directed cycle, contrary to acyclicity. Hence, this walk is  the directed $s$--$t$ path containing $(u,v)$.

Suppose that $\rbjt{J_Q}$ is not strong. Then there is a vertex $z$ from which some other vertex cannot be reached by a directed path in $\rbjt{J_Q}$. Let $U$ consist of $z$ and all vertices that can be reached from $z$ in $\rbjt{J_Q}$. Thus $U$ is nonempty and is not  $V(J)$. Moreover, no arc of $\rbjt{J_Q}$ is oriented from $U$ to $V(J)\setminus U$. 

We  show that $s\in U$. Suppose not. Since every vertex of $Q$ can reach $s$ by following the reversed arcs of $Q$ in $\rbjt{J_Q}$,  no vertex of $Q$ belongs to $U$. In particular, $t\notin U$. Choose $u\in U$. A directed $s$--$t$ path through $u$ in $\rbjt{J}$ contains a directed $u$--$t$ subpath. Since $t\notin U$, this subpath contains an arc from $U$ to $V(J)\setminus U$. Since the tail of this arc belongs to $U$, this arc does not belong to $Q$. It therefore has the same direction in $\rbjt{J_Q}$, contradicting that no arc of $\rbjt{J_Q}$ is oriented from $U$ to $V(J)\setminus U$. Thus $s\in U$.

We next show that $t\notin U$. If $t\in U$, then every vertex of $Q$ belongs to $U$. Choose $v\notin U$. A directed $s$--$t$ path through $v$ in $\rbjt{J}$ contains a directed $s$--$v$ subpath. Since $s\in U$ and $v\notin U$, this subpath contains an arc from $U$ to $V(J)\setminus U$. Since all vertices of $Q$ belong to $U$, this arc cannot belong to $Q$. Its direction is therefore unchanged in $\rbjt{J_Q}$, again a contradiction.

 Observe that no arc of $Q$ can be oriented from $V(J)\setminus U$ to $U$ in $\rbjt{J}$, because reversing such an arc would produce an arc from $U$ to $V(J)\setminus U$ in $\rbjt{J_Q}$. Now follow $Q$ from $s$ to $t$ in $\rbjt{J}$. Since $s\in U$ and $t\notin U$, 
  exactly one arc of $Q$,   denoted by $e$, is oriented from $U$ to $V(J)\setminus U$. Observe that any arc of $\rbjt{J}$ from $U$ to $V(J)\setminus U$ must belong to $Q$, since otherwise it would retain its direction in $\rbjt{J_Q}$. Hence, only $e$ is oriented from $U$ to $V(J)\setminus U$ in $\rbjt{J}$.

There is also no arc of $\rbjt{J}$ from $V(J)\setminus U$ to $U$. Suppose  $(u,v)$ were such an arc. By the first paragraph of this proof, this arc lies on a directed $s$--$t$ path in $\rbjt{J}$. Then the directed $s$--$u$ subpath must leave $U$, and hence must use $e$. And the directed $v$--$t$ subpath must also leave $U$, and hence must use $e$ again. This is impossible, since a directed $s$--$t$ path cannot use an edge twice.
 
By the above arguments, the underlying edge of $e$ is the only edge of $J$ between  $U$ and $V(J)\setminus U$, which is a bridge of $J$, a contradiction. Therefore $\rbjt{J_Q}$ is strong.

It remains to prove the distance estimate. Let $x,y$ be two distinct vertices. Choose a directed $x$--$y$ path $P$ in $\rbjt{J_Q}$. If an arc $(u, v)$ of $P$ is also an arc of $\rbjt{J}$, then  $1\le\varphi(v)-\varphi(u)$, where $1$ represents the length of $(u,v)$. If  $(u, v)$ is an arc of $P$ but $(v,u)$ is an arc of $Q$, then its length can be written as $1=\varphi(v)-\varphi(u)+[\varphi(u)-\varphi(v)+1]$, where $\varphi(u)-\varphi(v)+1\ge 2$.

Let $P=v_1v_2\cdots,v_r$, where $v_1=x$ and $v_r=y$. 
Then 
\[
\begin{aligned}
|P|\le  \sum_{i=1}^{r-1} \bigl(\varphi(v_{i+1})-\varphi(v_i)\bigr)+
   \sum_{(u,v)\text{ in }Q}
      \bigl(\varphi(v)-\varphi(u)+1\bigr)
= \varphi(y)-\varphi(x)+\varphi(t)-\varphi(s)+|Q|.
\end{aligned}
\]

Since $\dist_{\rbjt{J_Q}}(x,y)\le |P|$, the required bound follows.
\end{proof}

\begin{proof}[Proof of Lemma~\ref{lem:core}]
We first construct an auxiliary subgraph $K$ of $G$. Since $p\in X$, $q\in Y$, and $G$ is connected, there exists at least one crossing edge. For each crossing edge $uv$ with $u\in X$ and $v\in Y$, add $uv$ together with all shortest $p$--$u$ paths and all shortest $v$--$q$ paths to $K$. Obviously, $K$ is connected. The following gives some properties of $K$.

\begin{claim}\label{puvq}
  For each crossing edge $uv$ with $u\in X$ and $v\in Y$, we have the followings.

 (1)  All shortest $p$--$u$ paths and $v$--$q$ paths lie entirely in $X$ and $Y$, respectively.

    (2)      Let $p=x_0,x_1,\ldots,x_m=u$ be a shortest $p$--$u$ path. Then $a(x_i)=i$.
    
(3)  Let $v=y_0,y_1,\ldots,y_n=q$ be a shortest $v$--$q$ path. Then  $b(y_j)=n-j$.
\end{claim}

\begin{proof}
For each crossing edge $uv$ with $u\in X$ and $v\in Y$, let $p=x_0,x_1,\ldots,x_m=u$ be any shortest $p$--$u$ path. Then every its subpath from $p$ to $x_i$ is shortest, and hence $a(x_i)=i$. Since $x_{i-1}x_i$ is an edge, $b(x_i)\le b(x_{i-1})+1$. Therefore
\(
a(x_i)-b(x_i)\ge a(x_{i-1})-b(x_{i-1})
\) for every $1\le i\le m$. 
Since $u\in X$, it follows that every vertex of this $p$--$u$ path belongs to $X$.  
Similarly, let $v=y_0,y_1,\ldots,y_n=q$ be a shortest $v$--$q$ path. Then $b(y_i)=n-i$ and $a(y_i)\ge a(y_{i-1})-1$. Hence $a(y_i)-b(y_i)\ge a(y_{i-1})-b(y_{i-1})$ for $1\le i\le n$. Since $v\in Y$, every vertex of this $v$--$q$ path belongs to $Y$.
\end{proof}

Now we give an acyclic orientation $\rbjt{K}$ of $K$ with some properties, which will help us to give the orientation of $H$ and estimate the directed  distance bound in Lemma \ref{lem:core}.

\begin{claim}\label{clm:K-properties} There is an acyclic  orientation $\rbjt{K}$ of $K$ with the following properties.

(1) For each crossing edge $uv$ with $u\in X$ and $v\in Y$, all shortest $p$--$u$ paths and $v$--$q$ paths are oriented as a directed path from $p$ to $u$ and from $v$ to $q$, respectively. Every crossing edge is oriented from $X$ to $Y$.

(2) For every arc $(x,y)$ of $\rbjt{K}$, we have $a(y)-a(x)=1$ if $x,y\in X$, and $b(x)-b(y)=1$ if $x,y\in Y$.

(3) Every vertex and every arc of $\rbjt{K}$ lies on a directed $p$--$q$ path.

(4) There is a directed $p$--$q$ path of length $d$ in $\rbjt{K}$.
\end{claim}

\begin{proof}

We give the orientation $\rbjt{K}$ of $K$ as follows.
We orient every edge of $K$ inside $X$ towards the endpoint with larger $a$-value. Similarly,  we orient every edge within $Y$ towards the endpoint with smaller $b$-value. Finally,  every crossing edge is oriented from $X$ to $Y$. 

By the construction of $K$ and  Claim~\ref{puvq}, 
 $\rbjt{K}$ is   well defined and acyclic, and  we can see that (1)--(3) of this claim hold. Since $\dist_G(p,q)=d$, there is a $p$--$q$ path of length $d$. Indeed, this path must consist of a crossing edge $uv$,  a shortest $p$--$u$ subpath inside $X$ and a shortest $v$--$q$ subpath inside $Y$. By the above orientation, this path is oriented as a directed $p$--$q$ path of length $d$ in $\rbjt{K}$. 
\end{proof}

Before constructing $H$, we claim that there are at least two crossing edges and all crossing edges  are contained in the same block of $K$.
 Since $G$ is connected and bridgeless,
there are at least two crossing edges. 
 Fix a crossing edge $uv$ with $u\in X$ and $v\in Y$. Let $xy$ be any other  crossing edge, where $x\in X$ and $y\in Y$. By Claim \ref{puvq}(1) and the construction of $K$, we can find a $u$--$x$ path inside $X$ and $v$--$y$ path inside $Y$ in $K$. Then $uv$ and $xy$ lie in the same cycle in $K$. Hence,  all crossing edges  are contained in the same block of $K$.
 
 Now let $H$ be the block of $K$ containing all crossing edges. Then $H$ contains at least two vertices, and is connected and bridgeless. It remains to construct a strong orientation of $H$ satisfying \eqref{eq:core}. 
 Fix an orientation $\rbjt{K}$ supplied by Claim~\ref{clm:K-properties}. Firstly, we  find two vertices $s,t$ for which the restriction of $\rbjt{K}$ to $H$ satisfies the path conditions of Lemma~\ref{lem:reversal}.

\begin{claim}\label{clm:H-paths}
\emph{There are distinct vertices $s\in V(H)\cap X$ and $t\in V(H)\cap Y$ such that every vertex of $H$ lies on a directed $s$--$t$ path in the restriction of $\rbjt{K}$ to $H$. Moreover, one such directed $s$--$t$ path $\rbjt{Q}$ has length at most $d$.}
\end{claim}

\begin{proof}

Since $H$ is a block of $K$ and  contains all crossing edges, and since 
every $p$--$q$ path $P_0$ in $K$ must contain a crossing edge, the segment of $P_0$ in   $H$ is  a subpath of length at least one. By Claim~\ref{clm:K-properties}(4), choose a directed $p$--$q$ path $\rbjt{P}$ of length $d$ in $\rbjt{K}$, while $P$ is its underlying path. 
Denote the segment of  $P$ in $H$ as a   path $Q$, while $\rbjt{Q}$ is its orientation. Let $s$ and $t$ be the tail and head of $\rbjt{Q}$, respectively. Then $s\not=t$ and  $\rbjt{Q}$ is a directed $s$--$t$ path of length at most $d$. By the construction of $K$, Claim \ref{puvq}(1) and  Claim \ref{clm:K-properties}(1), $s\in V(H)\cap X$ and $t\in V(H)\cap Y$.

It remains to prove that every vertex of $H$ lies on a directed $s$--$t$ path in the restriction of $\rbjt{K}$ to $H$.
Let $\rbjt{P_0}$ be any other directed $p$--$q$ path in $\rbjt{K}$ with its underlying path ${P_0}$, and let $Q_0$ be the segment of $P_0$ in $H$ with its orientation $\rbjt{Q_0}$.
We first show that the tail and head of $\rbjt{Q_0}$ are $s$ and $t$, respectively. If the tail of $\rbjt{Q_0}$ is $s'\not=s$, then we can find a $s$--$s'$ path of length at least two inside $(X\backslash V(H))\cup \{s,s'\}$ in the subgraph $P\cup P_0$, a contradiction to $H$ being a block of $K$. Similarly, 
if the head of $\rbjt{Q_0}$ is $t'\not=t$, then we can find a $t$--$t'$ path of length at least two inside $(Y\backslash V(H))\cup \{t,t'\}$ in the subgraph $P\cup P_0$, a contradiction to $H$ being a  block of $K$. 
Hence, the tail and head of the segment of every directed $p$--$q$ path in $H$ are $s$ and $t$, respectively.

Fix any vertex $z\in V(H)$. By Claim~\ref{clm:K-properties}(3), there is a directed $p$--$q$ path through $z$ in $\rbjt{K}$. The tail and head of its segment in the restriction of $\rbjt{K}$ to $H$  are $s$ and $t$, respectively. Since $z\in V(H)$, this segment contains $z$. It is therefore a directed $s$--$t$ path through $z$ in the restriction of $\rbjt{K}$ to $H$. Hence,  every vertex of $H$ lies on a directed $s$--$t$ path in the restriction of $\rbjt{K}$ to $H$.
\end{proof}

Now we are ready to give the orientation of $H$. We first verify the condition on $\varphi$ in Lemma~\ref{lem:reversal}. Since the distances $a(x),b(x)$ and the numbers $d,\varepsilon$ are integers,  $\varphi$ is integer-valued. Let $(x,y)$ be an arc of the restriction of $\rbjt{K}$ to $H$. 
If $x,y\in X$, then $\varphi(x)=a(x)$ and $\varphi(y)=a(y)$. Claim~\ref{clm:K-properties}(2) therefore gives $\varphi(y)-\varphi(x)=a(y)-a(x)=1$.
If $x,y\in Y$, then $\varphi(x)=2d-\varepsilon-b(x)$ and $\varphi(y)=2d-\varepsilon-b(y)$.  Claim~\ref{clm:K-properties}(2) gives $b(x)-b(y)=1$. Therefore, $\varphi(y)-\varphi(x)=b(x)-b(y)=1$.

Finally,
by Claim~\ref{clm:K-properties}(1), 
 suppose that $x\in X$ and $y\in Y$.
 By the definition of $Y$, we have  $b(y)\le a(y)-\varepsilon-1$. Since every distance in $G$ is at most $d$,  $b(y)\le d-\varepsilon-1$ and $a(x)\le d$. It follows that
\[
\varphi(y)-\varphi(x)=2d-\varepsilon-b(y)-a(x)\ge 2d-\varepsilon-(d-\varepsilon-1)-d=1.
\]
Thus $\varphi(y)-\varphi(x)\ge1$ on every arc of the restriction of $\rbjt{K}$ to $H$.

We can now apply Lemma~\ref{lem:reversal}. Obviously, the restriction of $\rbjt{K}$ to $H$ is  acyclic. Claim~\ref{clm:H-paths} supplies distinct vertices $s,t$ such that every vertex of $H$ lies on a directed $s$--$t$ path in this restriction, together with such a path $Q$ of length at most $d$. Hence, all the hypotheses of Lemma~\ref{lem:reversal} are  satisfied. By Lemma~\ref{lem:reversal}, reversing exactly the arcs of $Q$ produces a strong orientation $\rbjt{H}$ of $H$, and every edge not in $Q$ retains its direction.

Since $s\in X$, we have $\varphi(s)=a(s)\ge0$. Since $t\in Y$, we have $\varphi(t)=2d-\varepsilon-b(t)\le2d-\varepsilon$. Consequently, $\varphi(t)-\varphi(s)\le2d-\varepsilon$. Using this inequality and $|Q|\le d$, we obtain
\[
\begin{aligned}
\dist_{\rbjt{H}}(u,v)
\le \varphi(v)-\varphi(u)+\varphi(t)-\varphi(s)+|Q|
\le \varphi(v)-\varphi(u)+3d-\varepsilon
\end{aligned}
\]
for all distinct $u,v\in V(H)$. If $u=v$, the directed distance is zero, while the right-hand side of~\eqref{eq:core} is $3d-\varepsilon>0$. Hence~\eqref{eq:core} holds for all $u,v\in V(H)$. 
\end{proof}

\section{Distance to the central subgraph}\label{sec:distance}

Retain the notation of Section~\ref{sec2}, and fix the subgraph $H$  supplied by Lemma~\ref{lem:core}. For $x\in V(G)$ and a strong orientation $\rbjt{G}$ of $G$, write $\dist_G(x,V(H))=\min_{v\in V(H)}\dist_G(x,v)$, $\dist_{\rbjt{G}}(x,V(H))=\min_{v\in V(H)}\dist_{\rbjt{G}}(x,v)$ and $\dist_{\rbjt{G}}(V(H),x)=\min_{v\in V(H)}\dist_{\rbjt{G}}(v,x)$. The following lemma shows that every vertex outside $H$ is within distance $\lfloor d/2\rfloor$ of $H$. This is the key estimate that allows Theorem~\ref{thm:radius} to be applied with $r=\lfloor d/2\rfloor$, and get the quadratic coefficient $1/2$.

\begin{lemma}\label{lem:distance-to-core}
For every vertex $x\in V(G)\backslash V(H)$, 
$
\dist_G(x,V(H))\le\left\lfloor\frac d2\right\rfloor.$

\end{lemma}

\begin{proof}
We first consider that $x\in X$. Let $P$ be a shortest $x$--$q$ path.
Since $x\in X$ and $q\in Y$, $P$ contains a crossing edge. By Lemma~\ref{lem:H}, every crossing edge belongs to $H$, so $P$ meets $H$. Let $v$ be the vertex in $H$ which is the first vertex of $P$ from $x$ to $q$ in $H$, and denote the $x$--$v$ subpath of $P$ by $P_0$.
Since $x\in X$ and all crossing edges belong to $H$, all  vertices of $P_0$ belong to $X$. In particular, $v\in X$.
Let $|P_0|=t$.
Since $P$ is a shortest $x$--$q$ path, $b(x)=t+b(v)$, and therefore $t=b(x)-b(v)\le d-b(v)$. Since $\dist_G(p,q)=d$, $d\le a(v)+b(v)$, so $d-b(v)\le a(v)$. Since $v\in X$, we also have $b(v)\ge a(v)-\varepsilon$, and hence $d-b(v)\le d+\varepsilon-a(v)$. Consequently, $t\le\min\{a(v),d+\varepsilon-a(v)\}$. By the  definition of $\varepsilon$,
\(
\dist_G(x,V(H))\le t\le\left\lfloor\frac{d+\varepsilon}{2}\right\rfloor=\left\lfloor\frac d2\right\rfloor.
\)

Similarly, consider $x\in Y$. Let $P$ be a shortest $x$--$p$ path. Since $x\in Y$ and $p\in X$, $P$ contains a crossing edge. By Lemma~\ref{lem:H}, every crossing edge belongs to $H$, so $P$ meets $H$. Let $v$ be the first vertex of $P$ in $H$ when $P$ is followed from $x$ to $p$, and denote the $x$--$v$ subpath of $P$ by $P_0$. Since $x\in Y$ and all crossing edges belong to $H$, all  vertices of $P_0$ belong to $Y$. In particular, $v\in Y$.
Let $|P_0|=t$. Since $P$ is a shortest $x$--$p$ path, $a(x)=t+a(v)$, and therefore $t=a(x)-a(v)\le d-a(v)$. Since $\dist_G(p,q)=d$, $d\le a(v)+b(v)$, so $d-a(v)\le b(v)$. Since $v\in Y$ and $a(v)-b(v)$ is an integer, we also have $a(v)-b(v)\ge\varepsilon+1$, and hence $d-a(v)\le d-\varepsilon-1-b(v)$. Consequently, $t\le\min\{b(v),d-\varepsilon-1-b(v)\}$.
Then  \(
\dist_G(x,V(H))\le t\le \left\lfloor\frac d2\right\rfloor.
\)
\end{proof}

\section{Proof of Theorem \ref{thm:main}}\label{sec:contraction}

We use the following specified-root form of the radius theorem of
Chv\'atal and Thomassen \cite{ChvatalThomassen1978}.  The exact statement is given by
Dankelmann, Morgan and Rivett-Carnac
\cite{DankelmannEtAl2024}, and follows from the proof of
Chv\'atal and Thomassen
\cite{ChvatalThomassen1978}.

\begin{theorem}[\cite{ChvatalThomassen1978,DankelmannEtAl2024}]\label{thm:radius}
Let $J$ be a finite connected bridgeless graph, let $z$ be a specified vertex, and let $r$ be a nonnegative integer. If $\dist_J(z,x)\le r$ for every vertex $x$, then $J$ has a strong orientation $\rbjt{J}$ for which both $\dist_{\rbjt{J}}(z,x)$ and $\dist_{\rbjt{J}}(x,z)$ are at most $r(r+1)$ for every vertex $x$.
\end{theorem}

\begin{proof}[Proof of Theorem \ref{thm:main}]
Fix $H$ and $\rbjt{H}$ as in Section~\ref{sec:distance}, and write $r=\lfloor d/2\rfloor$. We first contract $H$ to one vertex $z$: identify all vertices of $H$, delete every resulting loop, and retain all parallel edges as distinct edges. Denote the resulting graph by $J$. Obviously, the contraction operation keep $J$ to be connected and bridgeless.
Moreover, Lemma \ref{lem:distance-to-core} gives $\operatorname{dist}_J(z,x)\le\operatorname{dist}_G(x,V(H))\le r$ for every $x\in V(J)\setminus\{z\}$.

 Apply Theorem~\ref{thm:radius} and fix the resulting strong orientation $\rbjt{J}$. When recover $z$ from $J$ to obtain $G$, keep the orientations of all edges in $\rbjt{J}$.
 Keep the prescribed orientation $\rbjt{H}$ of $H$. Orient all edges which have both ends in $V(H)$ and do not belong to $E(H)$  arbitrarily. These  define an orientation $\rbjt{G}$ of $G$.

By  Theorem~\ref{thm:radius},  for any $x\in V(G)\backslash V(H)$, $\dist_{\rbjt{G}}(x,V(H))\le r(r+1)$ and $\dist_{\rbjt{G}}(V(H),x)\le r(r+1)$. By Lemma \ref{lem:core},
$\dist_{\rbjt{H}}(u,v)\le \varphi(v)-\varphi(u)+3d-\varepsilon$ for any $u,v\in V(H)$. Since $0\le\varphi(v)\le 2d$ for any $v\in V(H)$, we get $diam(\rbjt{H})\le 5d-\varepsilon$.
Then for any ordered pair $x,y\in V(G)$, 

 \[
\begin{aligned}
\dist_{\rbjt{G}}(x,y)&\le  dist_{\rbjt{G}}(x,V(H))+diam(\rbjt{H})+dist_{\rbjt{G}}(V(H),y) \\&\le 2r(r+1)+5d-\varepsilon \\& = 2 \lfloor d/2\rfloor^2+2\lfloor d/2\rfloor+5d-\varepsilon\\& \le \left\lceil\frac{d^2}{2}\right\rceil+6d.
\end{aligned}
\]
\end{proof}

\vskip 5mm
\noindent{\bf\large Acknowledgements}
\vskip 2mm

Chen is supported by National Key R\&D Program of China under grant number 2024YFA1013900 and NSFC under grant number 12471327, and 
Wang is supported by  NSFC under grant number 12401447. 
\vskip 5mm
\noindent{\bf\large AI Declaration}
\vskip 2mm

We used ChatGPT to assist with the exploration of proof strategies and the preparation of this manuscript.  Inspired by the contraction method used  in~\cite{BabuEtAl2021, WangChen2025}, in  September 2026,  we asked ChatGPT (GPT-6) whether the bound could be improved by contracting a suitable subgraph and applying induction on $d$. Its proposed argument did not use induction. Instead, it constructed a subgraph that is within a small distance of every  vertex outside this subgraph, and admits a strong orientation of diameter $O(d)$. And then extended this orientation to the whole graph. This argument gave a proposed upper bound of $\frac{2}{3}d^2+O(d)$.
When we asked whether the quadratic coefficient could be reduced further, ChatGPT did not provide a stronger bound.

To obtain the quadratic coefficient $1/2$, we were led to seek a connected bridgeless subgraph $H$ with the following two properties: every vertex outside $H$ is at distance at most $\lfloor d/2\rfloor$ from $H$, and $H$ admits a strong orientation of diameter $O(d)$. After contracting $H$ to a vertex $z$, Theorem~\ref{thm:radius} can be applied with $r=\lfloor d/2\rfloor$. Restoring the prescribed orientation of $H$ then adds only $O(d)$ to the directed distances, and hence yields an upper bound of the form $d^2/2+O(d)$.

To do this, we need to ensure that $H$ is a central subgraph of $G$.  
Inspired by the proofs for graphs of diameter~$3$ in~\cite{WangChen2022} by the first and third authors, where the vertex set is partitioned according to the distances from the two endpoints of a selected edge, we choose two vertices at distance $d$ and partition $V(G)$ into two parts $X$ and $Y$ according to their relative distances from these two vertices. During the development of the proof, we first considered a finer partition, but only two parts are needed in the final argument. We then construct $H$ so that it contains every edge between $X$ and $Y$. The resulting subgraph has precisely the two properties required above.

 In subsequent discussions about improving the linear term, ChatGPT proposed an argument with a linear coefficient strictly between $2$ and $3$, and we have checked its correctness. Hence the linear coefficient of $f(d)$ is between 1 and 3. This argument involved more complicated calculations but did not change the proof framework and   the subgraph $H$.
We believe that determining the sharp linear term may require either a better choice of the subgraph $H$  or a more general result in place of the radius theorem.

All mathematical arguments and proofs in the final
manuscript were written and verified by the authors.


\begin{thebibliography}{99}
\bibitem{AiLeiRen2026}
J.~Ai, H.~Lei and Z.~Ren,
A quadratic upper bound for the maximum oriented diameter with leading constant~1,
\emph{Journal of Graph Theory} (2026), Early View, 1--13.


\bibitem{BabuEtAl2021}
J.~Babu, D.~Benson, D.~Rajendraprasad and S.~N.~Vaka,
An improvement to Chv\'atal and Thomassen's upper bound for oriented diameter,
\emph{Discrete Applied Mathematics} \textbf{304} (2021), 432--440.


\bibitem{BauDankelmann2015}
S.~Bau and P.~Dankelmann,
Diameter of orientations of graphs with given minimum degree,
\emph{European Journal of Combinatorics} \textbf{49} (2015), 126--133.


\bibitem{ChenChang2023}
B.~Chen and A.~Chang,
Oriented diameter of graphs with given girth and maximum degree,
\emph{Discrete Mathematics} \textbf{346} (2023), no.~4, article~113287.


\bibitem{ChvatalThomassen1978}
V.~Chv\'atal and C.~Thomassen,
Distances in orientations of graphs,
\emph{Journal of Combinatorial Theory, Series B} \textbf{24} (1978), no.~1, 61--75.


\bibitem{CochranWang2026}
G.~Cochran and Z.~Wang,
On the oriented diameter of graphs with given minimum degree,
\emph{Journal of Graph Theory} (2026), Early View, 1--14.


\bibitem{DankelmannGuoSurmacs2018}
P.~Dankelmann, Y.~Guo and M.~Surmacs,
Oriented diameter of graphs with given maximum degree,
\emph{Journal of Graph Theory} \textbf{88} (2018), no.~1, 5--17.


\bibitem{DankelmannEtAl2024}
P.~Dankelmann, J.~Morgan and E.~Rivett-Carnac,
The oriented diameter of graphs with given connected domination number and distance domination number,
\emph{Graphs and Combinatorics} \textbf{40} (2024), article~18.




\bibitem{FengEtAl2026}
Y.~Feng, H.~Lei, X.~Lian and Z.~Ren,
An improved upper bound for oriented diameter of graphs with diameter~$4$,
\emph{arXiv preprint arXiv:2609.29526}, 2026.

\bibitem{FominEtAl2004}
F.~V.~Fomin, M.~Matamala, E.~Prisner and I.~Rapaport,
AT-free graphs: linear bounds for the oriented diameter,
\emph{Discrete Applied Mathematics} \textbf{141} (2004), nos.~1--3, 135--148.


\bibitem{FominMatamalaRapaport2004}
F.~V.~Fomin, M.~Matamala and I.~Rapaport,
Complexity of approximating the oriented diameter of chordal graphs,
\emph{Journal of Graph Theory} \textbf{45} (2004), no.~4, 255--269.


\bibitem{GeLiuWang2025}
Y.~Ge, X.~Liu and Z.~Wang,
On the oriented diameter of near planar triangulations,
\emph{Discrete Mathematics} \textbf{348} (2025), no.~5, article~114406.



\bibitem{G} G. Gutin, Minimizing and maximizing the diameter in orientations of graphs, \emph{Graphs and Combinatorics}, \textbf{10}(1994), 225--230.



\bibitem{KurzLatsch2012}
S.~Kurz and M.~L\"atsch,
Bounds for the minimum oriented diameter,
\emph{Discrete Mathematics \& Theoretical Computer Science} \textbf{14} (2012), no.~1, 109--140.


\bibitem{KwokLiuWest2010}
P.~K.~Kwok, Q.~Liu and D.~B.~West,
Oriented diameter of graphs with diameter~3,
\emph{Journal of Combinatorial Theory, Series B} \textbf{100} (2010), no.~3, 265--274.


\bibitem{LinWangYou2026}
J.~Lin, X.~Wang and L.~You,
An improved upper bound on the oriented diameter of graphs with diameter~$4$,
\emph{arXiv preprint arXiv:2605.11667}, 2026.

\bibitem{LinYou2026}
J.~Lin and L.~You,
Oriented diameter of graphs with diameter~$4$ and given edge girth,
\emph{Discrete Applied Mathematics} \textbf{391} (2026), 413--431.


\bibitem{Liu2026}
X.~Liu,
Improved bounds on the oriented diameter of planar triangulations,
\emph{arXiv preprint arXiv:2608.19089}, 2026.


\bibitem{P} J. Plesn\'ik, Remarks on diameters of orientations of graphs, \emph{Acta Mathematica Universitatis Comenianae}, \textbf{46-47}(1985), 225--236.


\bibitem{Robbins1939}
H.~E.~Robbins,
A theorem on graphs, with an application to a problem of traffic control,
\emph{The American Mathematical Monthly} \textbf{46} (1939), no.~5, 281--283.



\bibitem{LS} L. $\check{\text{S}}$olt\'es, Orientations of graphs minimizing the radius or the  diameter, \emph{Mathematica Slovaca}, \textbf{36}(1986), 289--296.

\bibitem{Surmacs2017}
M.~Surmacs,
Improved bound on the oriented diameter of graphs with given minimum degree,
\emph{European Journal of Combinatorics} \textbf{59} (2017), 187--191.


\bibitem{Wang} X.~Wang, Oriented diameter and radius of triangulations, \emph{Annals of Applied Mathematics}, \textbf{41}(2025), 484--491.



\bibitem{WangChen2022}
X.~Wang and Y.~Chen,
Optimal oriented diameter of graphs with diameter~3,
\emph{Journal of Combinatorial Theory, Series B} \textbf{155} (2022), 374--388.


\bibitem{WangChen2025}
X.~Wang and Y.~Chen,
Oriented diameter of graphs with given domination number,
\emph{arXiv preprint arXiv:2506.15997}, 2025. 

\bibitem{WangEtAl2021}
X.~Wang, Y.~Chen, P.~Dankelmann, Y.~Guo, M.~Surmacs and L.~Volkmann,
Oriented diameter of maximal outerplanar graphs,
\emph{Journal of Graph Theory} \textbf{98} (2021), no.~3, 426--444.

\end{thebibliography}
\end{document}